\documentclass[10pt,a4paper]{article}
\usepackage[latin1]{inputenc}
\usepackage[english]{babel}
\usepackage{latexsym}
\usepackage{amsfonts}\usepackage{amsmath}

\newtheorem{theorem}{Theorem}[section]
\newtheorem{proposition}[theorem]{Proposition}
\newtheorem{lemma}[theorem]{Lemma}
\newtheorem{corollary}[theorem]{Corollary}
\newtheorem{definition}[theorem]{Definition}
\newtheorem{remark}[theorem]{Remark}

\newtheorem{example}{Example}

\newenvironment{proof}{\noindent{\it Proof.\ }}{\hfill $\square$\vspace{2mm}}

\sectionmark \leftmark

\newcommand{\R}{\mathbb{R}}

\newcommand{\E}{\mathbb{E}}
\newcommand{\h}{\mathcal{H}}

\newcommand{\pr}{\mathbb{P}}
\newcommand{\1}{\mathbf{1}}

\newcommand{\D}{{\rm d}}

\newcommand{\bd}{\begin{definition}}
\newcommand{\ed}{\end{definition}}
\newcommand{\bt}{\begin{theorem}}
\newcommand{\et}{\end{theorem}}
\newcommand{\bp}{\begin{proposition}}
\newcommand{\ep}{\end{proposition}}
\newcommand{\br}{\begin{remark}}
\newcommand{\er}{\end{remark}}
\newcommand{\bl}{\begin{lemma}}
\newcommand{\el}{\end{lemma}}
\newcommand{\bc}{\begin{corollary}}
\newcommand{\ec}{\end{corollary}}
\newcommand{\beq}{\begin{equation}}
\newcommand{\eeq}{\end{equation}}

\begin{document}
\baselineskip=13pt

\title{A note on mean volume and surface densities for a class of  birth-and-growth stochastic processes}
\author{Elena Villa}
\date{\small Dept. of Mathematics, University of Milan, via Saldini 50,
20133 Milano, Italy\\
email: elena.villa@mat.unimi.it}

\maketitle

\begin{abstract}

Many real phenomena may be modelled as locally finite unions of
$d$-dimensional time dependent random closed sets in
$\mathbb{R}^d$, described by birth-and-growth stochastic
processes, so that their mean volume and surface densities, as
well as the so called mean \emph{extended} volume and surface
densities, may be studied in terms of relevant quantities
characterizing the process. We extend here known results in the
Poissonian case to a wider class of birth-and-growth stochastic
processes, proving in particular the absolute continuity of the
random time of capture of a point  $x\in\R^d$ by processes of this
class.
\end{abstract}
Keywords: Stochastic  birth-and-growth processes; random closed
sets; mean densities\\
 AMS Classification 2000: 60D05;60G55,28A75\\

\section{Problem and main results}
A great variety of real phenomena  in material science and in
biomedicine, such as crystallization processes (see \cite{libro},
and references therein; see also \cite{ubukata} for the
crystallization processes on sea shells), tumor growth
\cite{anderson,KM}, spread of fires in the woods, etc., can be
described by space-time structured stochastic birth-and-growth
processes (see, e.g., \cite{KV-surv}).
  Roughly
speaking, a \emph{birth-and-growth (stochastic) process} is a
dynamic germ-grain model \cite{stoyan,Weil}, used to model
situations in which \emph{nuclei} are born in time and are located
in space randomly, and each nucleus generates a \emph{grain} (a
random closed set) evolving in time. So it can be
 described by a
 marked point
process $N =\{(T_j,X_j)\}_{j\in\mathbb{N}}$ modelling births at
random times   $T_j \in \mathbb{R}_+$ and related random spatial
locations (\emph{nuclei})  $X_j \in \mathbb{R}^d$ ($d\geq 2)$, and
by a growth model according to which each  nucleus generates a
grain $\Theta_{T_j}^t(X_j)$ evolving in time. Under regularity
assumptions on the birth and growth model, the   union set
$\Theta^t$ of such grains at time $t$ is then a locally finite
union of random closed sets and   the mean  volume and surface
densities associated to the birth-and-growth process
$\{\Theta^t\}_t$  can be defined. Sometimes it is of interest to
consider the so-called mean \emph{extended} densities of
$\Theta^t$, defined as the mean densities of the union of the
grains $\Theta_{T_j}^t(X_j)$ ignoring overlapping; for instance
the mean density of the $d-1$-dimensional measure of the union of
the topological boundaries $\partial\Theta_{T_j}^t(X_j)$ might be
studied whenever the process $\{\Theta^t\}_t$ is given by the
union of  $(d-1)$-dimensional grains free to grow in space.  A
natural question is whether any relationship exists between these
densities and, in particular, if it is possible to describe them
in terms of relevant quantities associated with the process, like
the intensity measure of the nucleation process $N$ and the growth
rate.  In  current literature, the particular case in which $N$ is
given by a marked Poisson process has been studied extensively,
and great importance has been given to the concept of \emph{causal
cone} and its relationship with the mean (extended) volume density
(e.g., see \cite{BKM-hazard,BKP,libro,kolm2}). In particular a
relationship between the measure of the causal cone with respect
to the intensity measure of the nucleation process and the mean
extended volume density  has been proven in \cite{BKM-hazard},
where the property of independence of the grains, due to the
Poisson assumption, plays a fundamental role. Since such
quantities are well defined also for more general birth-and-growth
processes, aim of the present paper is to extend known results in
the Poissonian case to a wider family of processes. To this end we
introduce here a class of birth-and-growth processes, denoted by
$\cal G$,  satisfying quite general assumptions, and we show that
the quoted result on the mean extended volume density (see
Proposition~\ref{teo vex}) and, in particular, an equation for the
mean extended surface density (see Proposition~\ref{corollary
s-ex}) hold for any process in $\cal G$. In particular, in order
to do this, we prove that the so called time of capture $T(x)$ of
a point $x\in\R^d$ associated with a process
$\{\Theta^t\}_t\in\cal G$ is a continuous random variable with
density (see Theorem~\ref{teo T}). Examples of non-Poissonian
birth-and-growth processes in $\cal G$ are also provided.

\section{Preliminaries and notations}\label{sec notations}
We recall that a \emph{random closed set} $\Theta$ in
$\mathbb{R}^d$ is a measurable map
$\Theta:(\Omega,\mathfrak{F},\mathbb{P})\longrightarrow(\mathbb{F},\sigma_\mathbb{F}),$
where $\mathbb{F}$ denotes the class of the closed subsets in
$\mathbb{R}^d$, and $\sigma_\mathbb{F}$ is the $\sigma$-algebra
generated by the so called \emph{hit-or-miss topology} (see
\cite{matheron}).  Denoted  by  $T_j$ the $\mathbb R_+$-valued
random variable representing the \emph{time of birth} of the
$j$-th nucleus, and by $X_j$ the $\mathbb{R}^d$-valued random
variable representing the \emph{spatial location} of the nucleus
born at time $T_j$, defined on the same probability space, let
$\Theta^t_{T_j}(X_j)$
 be the random closed set obtained as the evolution up to time $t \geq T_j$ of the
 nucleus
born at time $T_j$ in $X_j$. The family  $\{\Theta^t\}_t$ of
random closed sets
  given by $$\Theta^{t}= \bigcup_{T_j \leq t}
\Theta^t_{T_j}(X_j), \qquad t \in {\mathbb R_+},$$ is called {\em
birth-and-growth (stochastic) process}. The \emph{nucleation
process} $\{(T_j,X_j)\}$ is usually described by a marked point
process (MPP) $N$ in $\mathbb{R}_+$ with marks in $\mathbb{R}^d$.
(For basic definitions and results about MPPs we refer to
\cite{daley,last-brandt,stoyan}). Thus, it is defined as a random
measure given by
$$N=\sum_{j=1}^\infty \delta_{T_j,X_j},$$  where
$\delta_{t,x}$ denotes here the Dirac measure on
$\mathbb{R}_+\times\mathbb{R}^d$ concentrated at $(t,x)$; so, for
any  $B\times A\in{\cal B}_{\mathbb{R}}\times {\cal
B}_{\mathbb{R}^d}$ (${\cal B}_{\mathbb{R}^d}$ is the Borel
$\sigma$-algebra of $\R^d$), $ N(B \times A)$ is the random number
of nuclei born in the region $A$, during time $B$. We recall that
the \emph{marginal process} $\widetilde{N}(\cdot):=N(\,\cdot\times
\mathbb{R}^d)$ is itself a point process. Throughout the paper we
denote by $\Lambda$ and $\widetilde\Lambda$ the \emph{intensity
measure} of $N$ and $\widetilde N$, respectively, so defined
$$\Lambda(B\times A):=\mathbb{E}[N(B\times A)],\qquad \widetilde\Lambda(B):=\mathbb{E}[\widetilde{N}(B)],
\qquad A\in\mathcal{B}_{\R^d},\, B\in\mathcal{B}_{\R};$$ the
measure $\widetilde\Lambda$ is usually assumed to be locally
finite, and it is well known the following decomposition of the
intensity measure  (see, e.g., \cite{last-brandt}): \beq\Lambda(\D
t\times \D x)=\widetilde{\Lambda}(\D t)Q(t,\D
x),\label{lambda}\eeq where, $\forall t\in\mathbb{R}_+$,
$Q(t,\cdot)$ is a probability measure on $\mathbb{R}^d$, called
the \emph{mark distribution} at time $t$.

 Models of volume growth have been studied extensively,
since the pioneering work by Kolmogorov \cite{kolm2} (see also
\cite{BKM-hazard}). We denote by $\h^n$ the $n$-dimensional
Hausdorff measure and recall that $\h^d(B)$ coincides with the
usual $d$-dimensional Lebesgue measure of $B$ for any Borel set
$B\subset\R^d$, while for $1\leq n < d$ integer $\h^n(B)$
coincides with the classical $n$-dimensional measure of $B$ if $B$
is contained in a $C^1$ $n$-dimensional manifold embedded in
$\R^d$. Throughout the paper we assume $d\geq 2$ and the
\emph{normal growth model} (see, e.g., \cite{BKP}), according to
which at $\mathcal{H}^{d-1}$-almost every point of the actual
grain surface at time $t$ (i.e.~at $\mathcal{H}^{d-1}$-a.e. $x\in
\partial \Theta^t_{T_j}(X_j)$) growth occurs with a given strictly positive normal
velocity \begin{equation} v(t,x)=
G(t,x)n(t,x),\label{model}\end{equation}
 where $G(t,x)$   is a given deterministic
 \emph{growth  field}, and $n(t,x)$ is the unit outer normal
 at point  $x \in \partial\Theta^t_{T_0}(X_0)$. We assume that
  $$0<g_0\leq G(t,x)\leq G_0<\infty\qquad
 \forall (t,x)\in\mathbb{R}_+\!\times\mathbb{R}^d ,$$ for some
 $g_0,G_0\in\mathbb{R}$, and  $G(t,x)$ is sufficient regular
  such that the evolution problem given by
 (\ref{model})
   for the growth front $\partial\Theta_{t_0}^t(x)$ is well
posed. It follows that for any fixed $t\in\R_+$, the topological
boundary of each grain is a random closed set with locally finite
$\h^{d-1}$-measure $\pr$-almost surely (see also
\cite{burger2002}). Furthermore, for the birth-and-growth model
defined above,
 the so-called {\em causal cone} associated with a point $x\in\mathbb{R}^d$ and a time $t\in\mathbb{R}_+$
is well defined (see e.g. \cite{BKP} for its analytical
properties).
\begin{definition}[Causal cone]\label{def cono}
The {\em causal cone} $\mathcal{C}(t,x)$  of a point $x$ at time
$t$ is the space-time region in which at least one nucleation has
to take place so that the point $x$ is covered by grains by time
$t$: $$\mathcal{C}(t,x):=\{(s,y)\in[0,t]\times\
\mathbb{R}^d\,:\,x\in\Theta_s^t(y)\}.$$
\end{definition}

To any point $x\in\R^d$ it  is also associated a random variable
$T(x)$, said the \emph{time of capture of point $x$}, defined by
\beq T(x) :=\min\{t>0\,:\,x\in\Theta^t\}.\label{T}\eeq

 We know that any random closed set
$\Theta$ in $\R^d$ with locally finite $\h^n$ measure $\pr$-a.s.,
induces a random measure
$\mu_\Theta(\cdot):=\mathcal{H}^n(\Theta\cap \,\cdot\,)$ on $\R^d$
(for a discussion of the  measurability of the random variables
$\mathcal{H}^n(\Theta\cap \,\cdot\,)$, we refer to
\cite{MB,Zahle}), and  it is clear that $\mu_{\Theta(\omega)}$ is
 singular with respect to $\h^d$ for $\pr$-a.e. $\omega\in\Omega$ if $n<d$. On the
other hand, the expected measure
$\mathbb{E}[\mu_{\Theta}](\cdot):=\mathbb{E}[\mathcal{H}^n(\Theta\cap
\cdot)]$ may be absolutely continuous with respect to $\h^d$,  in
dependence of the probability law of $\Theta$; in such case the
random closed set $\Theta$ is said to be \emph{absolutely
continuous in mean} (see \cite{KV-ias}). \\
For any fixed $t\in\R_+$ the following measures on $\R^d$
associated to a birth-and-growth process $\{\Theta^t\}_t$ as
above, and their respective densities (provided that the
topological boundary $\partial\Theta_{T_j}^t(X_j)$ of each grain
$\Theta_{T_j}^t(X_j)$ is absolutely continuous in mean), can be
introduced (see \cite{BKM-hazard,KM}):
\begin{definition}[Mean volume and surface measures and densities]\label{def measures} For
any $t\in\R_+$
\begin{itemize}
\item the measure
$\mathbb{E}[\mu_{\Theta^t}](\,\cdot\,):=\mathbb{E}[\h^d(\Theta^t\cap\cdot\,)]$
on $\mathbb{R}^d$ is said
 \emph{mean volume measure} at time $t$, while the quantity $V_V(t,x)$ such that, for any
$A\in\mathcal{B}_{\mathbb{R}^d}$,
$$\mathbb{E}[\mu_{\Theta^t}]( A)]=\int_AV_V(t,x)\D x,$$ is
called \emph{mean volume density} (or \emph{crystallinity}) at
point $x$ and time $t$ ($\D x$ stands for $\h^d(\D x)$); \item the
measure $\E[\mu_{\Theta^t}^{\rm
ex}](\,\cdot\,):=\mathbb{E}[\sum_{j: T_j\leq
t}\mathcal{H}^d(\Theta_{T_j}^t(X_j)\cap\,\cdot\,)]$ on $\R^d$ is
said \emph{mean extended volume measure} at time $t$, while the
quantity $ V_{\rm ex}(t,x)$  such that,  for any  $A \in {\cal
B}_{\mathbb R^d}$,
 $$\E[\mu_{\Theta^t}^{\rm
ex}]( A)
 = \int_A V_{\rm ex}(t,x)  \D x,$$
is called {\em mean extended volume density} at point $x$ and time
$t$; \item the measure
$\mathbb{E}[\mu_{\partial\Theta^t}](\,\cdot\,):=\mathbb{E}[\mathcal{H}^{d-1}(\partial\Theta\cap\,\cdot\,)]$
on $\R^d$ is said  \emph{mean surface measure} at time $t$, while
the quantity  $S_V(t,x)$ such that, for any
$A\in\mathcal{B}_{\mathbb{R}^d}$,
$$\mathbb{E}[\mu_{\partial\Theta^t}](A)=\int_AS_V(t,x)\D x, $$
is called  \emph{mean surface density} at point $x$ and time $t$;
\item the   measure $\E[\mu_{\partial\Theta^t}^{\rm
ex}](\,\cdot\,):=\mathbb{E}[\sum_{j: T_j\leq t}\mathcal{H}^{d-1}
(\partial\Theta_{T_j}^t(X_j)\cap\,\cdot\,)]$ on $\mathbb{R}^d$ is
said \emph{mean extended surface measure} at time $t$, while the
quantity
 $S_{\rm ex}(t,x)$  such that, for any $A \in  {\cal
B}_{\mathbb R^d}$,
$$\E[\mu_{\partial\Theta^t}^{\rm
ex}](A)
 = \int_B S_{\rm ex}(t,x)  \D x,$$ is called {\em mean extended surface density} at point $x$ and time
 $t$.
\end{itemize}
\end{definition}
In other words, the mean extended volume and surface measures
represent the mean of the sum of the volume measures  and of the
surface measures of the grains which are born and grown until time
$t$, supposed \emph{free to grow}, ignoring overlapping. Note that
in the particular case in which  $\Theta^t$ is stationary,
$V_V(t,\cdot)$ and $S_V(t,\cdot)$ are constant and they are said
\emph{volume fraction} and \emph{surface density} of $\Theta^t$,
respectively (see, e.g., \cite{stoyan}, p.\,342).\\
We mentioned that a  problem of interest in real applications is
to find relationships
 about the above mean densities, being relevant quantities describing the geometric process
 $\{\Theta^t\}_t$.  Recent results in this direction show that, if $G(t,x)$ is
such that the topological boundary of the grains satisfies a
certain regularity condition (related to rectifiability
properties), then
  an evolution equation for the mean
densities can be proved; namely (see Proposition~25 in
\cite{KV-dens}), \begin{proposition}\label{prop evol} Let
$\{\Theta^t \}_t$ be a birth-and-growth process with growth model
as above such that:
\begin{itemize}
\item the marginal process $\widetilde N$ is such that
$\E[\widetilde{N}([t,t+\Delta t])\1_{\widetilde{N}([t,t+\Delta
t])\geq 2}]=o(\Delta t)$ $\forall t>0$; \item $\forall t>0$,
denoted by $\Theta^t_r$ the closed $r$-neighborhood of $\Theta^t$
(i.e.~$ \Theta^t_r:=\{x\in\mathbb{R}^d:\,\text{$\exists y\in
\Theta^t$ with $|x-y|\leq r$}\}$), the following limit holds for
any bounded Borel set $A$ with $\h^d(\partial A)=0$
\beq\lim_{r\downarrow 0}\frac{\E[\h^d((\Theta^t_r\setminus
\Theta^t)\cap A)]}{r}=\E[\h^{d-1}(\partial \Theta^t\cap
A)];\label{eq Mink}\eeq \item the time of capture $T(x)$ is a
continuous random variable with density.
\end{itemize}
Then the following evolution equation holds in weak form
\begin{equation}
 \frac{\partial}{\partial t}V_{V}(t,x)= G(t,x) S_{V}(t,x). \label{vv}
 \end{equation}

\end{proposition}
(For a discussion about equation \eqref{eq Mink} see \cite{AColeV,KV-dens}.)\\
Note that,  by linearity arguments,  an analogous relationship for
the mean extended densities follows:
\begin{equation}
 \frac{\partial}{\partial t}V_{\rm ex}(t,x)= G(t,x) S_{\rm ex}(t,x), \label{DVS}
 \end{equation}
to be taken, as usual,  in weak form.\\ While it is easily seen
that $V_V(t,x)=\pr(x\in\Theta^t)$ for $\h^d$-a.e.\,\,$x\in\R^d$
(see Section~\ref{subsec}), an analogous result about $V_{\rm
ex}$,
 and so about $S_{\rm ex}$ by \eqref{DVS}, is    known in current literature only in the case of Poisson type nucleation
 processes. Namely, it has been shown in \cite{BKM-hazard}, Theorem~1, that if $N$ is
a marked Poisson process with intensity measure  $\Lambda(\D
t\times\D x)=\alpha(t,x)\D t\D x$ with $\alpha$ such that
$\widetilde{\Lambda}([0,t])<\infty$ and $\alpha(t,\cdot\,)\in
L^1(\R^d)$ $\forall t\in\R_+$, then \begin{equation} V_{\rm
ex}(t,x)=\Lambda(\mathcal{C}(t,x)) \label{nu-vex}\end{equation}
for $\h^1\!\times\h^d\mbox{-a.e.}~(t,x)\in\R_+\times\R^d.$ \\ By
Definition~\ref{def cono} it follows that
$V_V(t,x)=\mathbb{P}(N(\mathcal{C}(t,x))>0), $ therefore whenever
the nucleation process $N$ is given by a marked Poisson process
with intensity measure $\Lambda$ as above we have that
\beq\frac{\partial}{\partial
t}V_V(t,x)=(1-V_V(t,x))\frac{\partial}{\partial t} V_{\rm
ex}(t,x).\label{VV-Vex}\eeq

 In the next section we will show that, while
Eq.~\eqref{VV-Vex} is true only in the Poissonian case thanks to
the property of independence of increments which characterizes
Poisson processes, Eq.~\eqref{nu-vex}, and consequently a formula
for the mean extended surface density $S_{\rm ex}$  by
\eqref{DVS}, holds for a wider class of birth-and-growth
processes, which can be taken as model in various real
applications.

\section{Extensions to the non-Poissonian case}\label{sec non-poisson}
\subsection{A class of birth-and-growth stochastic processes}
\bd[The class $\mathcal{G}$] Let $\cal G$ be the family of all
birth-and-growth processes $\{\Theta^t\}_t$ with growth model as
above such that $\Theta^t$ satisfies equation \eqref{eq Mink} for
any  $t\in\R_+$ and  the following assumptions on the nucleation
process $N$ are fulfilled:
\begin{itemize}\item[(A1)]
$\E[\widetilde{N}([t,t+\Delta t])\1_{\widetilde{N}([t,t+\Delta
t])\geq 2}]=o(\Delta t)$ for all $t>0$;
 \item[(A2)] with  respect to the decomposition of $\Lambda$ in \eqref{lambda},
 $\widetilde\Lambda$ is locally finite with density $\lambda$, and
 for all $t>0$
the mark distribution $Q(t,\cdot\,)$ admits density $q(t,\cdot\,)$
on $\R^d$.
\end{itemize}
\ed A few comments about the assumptions defining the class $\cal
G$: \begin{itemize} \item  Condition (A1) is closely related to
the notion of \emph{simple} point process (see, e.g., \cite{karr})
and it is used in the proof of the evolution equation \eqref{vv}.
Besides, observing that $$\pr(\widetilde{N}([t,t+\Delta t])\geq
2)\leq\sum_{n=2}^\infty n\pr(\widetilde{N}([t,t+\Delta
t])=2)\stackrel{(A1)}{=}o(\Delta t),
$$ it guarantees that for any infinitesimal time interval $\Delta
t$ at most one nucleation can occur,
i.e.~\beq\pr(\widetilde{N}([t,t+\Delta t])>
0)=\pr(\widetilde{N}([t,t+\Delta t])=1)+ o(\Delta t),\label{prob
o}\eeq which is usually assumed in modelling many real situations.
\item Condition (A2) implies that
$\widetilde{\Lambda}([0,t])<\infty$, which is a
 common
assumption  in the theory of point processes, and, in particular,
that  the  intensity measure $\Lambda$ is absolutely continuous
with respect to $\h^1\times\h^d$, by \eqref{lambda}. This,
together with the growth model assumptions, guarantees that the
boundary of each grain is absolutely continuous in mean, so that
the  mean surface density $S_V$ and the mean extended surface
density $S_{\rm ex}$ are well defined.\\
 Further, denoted by $\mathcal{S}_x(s,t):=\{y\in \mathbb{R}^d:(s,y)\in\mathcal{C}(t,x)\}$
  the section of the causal
cone $\mathcal{C}(t,x)$ at time $s<t$ and  $S(y;x,t):=\sup\{s\geq
0\,|\,(y,s)\in\mathcal{C}(t,x)\}$ if
 $(y,0)\in\mathcal{C}(t,x)$,
Proposition~4.1 in \cite{BKP} ensures that, if  $\Lambda$ is
absolutely continuous with respect to $\h^1\times\h^d$ with
density $\alpha(t,x)$, then $\Lambda(\mathcal{C}(t,x))$ is
continuously differentiable with respect to $t$ and, in
particular,
\begin{equation} \frac{\partial}{\partial t} \Lambda(\mathcal{C}(t,x))=
G(t,x) \int_0^t  \int_{\partial {\cal S}_x(s,t)}  \alpha (s,y)\,
\D K_{x,t,s}(y)\,\D s,\label{evolvex}\end{equation} with the
measure \beq\D
K_{x,t,s}(y)=\frac{|\nabla_xS|(y;x,t)}{|\nabla_yS|(y;x,t)}\D\h^{d-1}(y).\label{K
meas}\eeq So, by condition (A2), we have that \eqref{evolvex}
holds for any process in $\cal G$ with
$\alpha(s,y)=\lambda(s)q(s,y)$.
\end{itemize}

Now we show by simple examples  that the class $\mathcal{G}$ is
not trivial and strictly contains the birth-and-growth processes
with Poissonian nucleation process.

\begin{proposition}\label{prop poisson} Let $\{\Theta^t\}_t$ be a birth-and-growth process
 with $G(t,x)$ sufficiently regular as in previous
assumptions and $N$ marked Poisson process with intensity measure
$\Lambda$ satisfying condition (A2). Then $\{\Theta^t\}_t\in\cal
G$.
\end{proposition}
\begin{proof} By the well known definition of marked Poisson point process
 we have that
the marginal process $\widetilde{N}$ is a Poisson process with
intensity measure $\widetilde{\Lambda}(\D t)=\lambda(t)\D t$, by
assumption (A2). So we have to prove only condition (A1).
 Let $t\in\mathbb{R}_+$ be fixed. Recalling the Poisson property
 $\pr(\widetilde{N}([t,t+\Delta t])\geq 1)=\widetilde{\Lambda}([t,t+\Delta t])+o(\widetilde{\Lambda}([t,t+\Delta
 t]))$,
 we have that
\begin{multline*}
\mathbb{E}[\widetilde{N}([t,t+\Delta
t])\1_{\widetilde{N}([t,t+\Delta t])\geq 2}]=\sum_{n=2}^\infty
n\frac{\widetilde{\Lambda}([t,t+\Delta
t])^n}{n!}e^{-\widetilde{\Lambda}([t,t+\Delta t])}\\=
\widetilde{\Lambda}([t,t+\Delta t])\pr(\widetilde{N}([t,t+\Delta
t])\geq 1)=o(\Delta t).
\end{multline*}
\end{proof}

\begin{example}\label{esempio 1}{\rm
 Let $\{\Theta^t\}_t$ be a birth-and-growth process with $G(t,x)$
 sufficiently
 regular as in previous assumptions and nucleation process $N^{(1)}$  given by the birth of
only one nucleus $(T,X)$ with $T\geq 0$ continuous
 random variable
with density and $X$ random point in $\R^d$ with distribution
$Q\ll\h^d$. Clearly $\{\Theta^t\}_t\in {\cal G}$, and for any
 $t\in\R_+$, $ V_{\rm ex}(t,x)=V_V(t,x)$ $\h^d$-a.e.
$x\in\R^d,$ since $\Theta^t=\Theta^t_T(X)$ (with
$\Theta_T^t(X)=\emptyset$ if $T>t$).}
\end{example}
In the next example we provide a non-trivial (i.e.~like $N^{(1)}$)
birth-and-growth process  belonging to the class $\cal{G}$, with
nucleation process $N$ which is not given by a marked Poisson
process.

\begin{example}\label{esempio non poisson}{\rm
Let  $G(t,x)$ be sufficiently regular as in previous assumptions,
and let $T_1\geq 0$ be a continuous random variable with
probability density function $f$.  We assume that the first
nucleus  is born at the random time $T_1$ and that a new
nucleation occurs at times $T_1+1,T_1+2,\ldots$
(i.e.~$T_j=T_1+j-1$). Let the spatial locations $X_1,X_2,\ldots$
of the nuclei be IID and independent of $T_1$, with distribution
$Q\ll\h^d$. It follows that

$\mathbb{P}(\widetilde{N}([0,t])=0)=\mathbb{P}(T_1>t) $,

$\mathbb{P}(\widetilde{N}([0,t])=n)=\mathbb{P}(t-n<T_1\leq
t-n+1),\quad \mbox{for } n=1,2,\ldots,[t]+1$,

$\mathbb{P}(\widetilde{N}([0,t])=n)=0,\quad \mbox{for }n>[t]+1 $,
\\
where $[t]$ is the integer part of $t$. As a consequence,
$\widetilde{N}([0,t])\leq [t]+1$ $\pr$-a.s. and
$$\widetilde{\Lambda}([0,t])=\sum_{n=1}^{[t]+1}n\mathbb{P}(t-n<T_1\leq t-n+1)
=\sum_{j=0}^{[t]}\pr(T_1\leq
t-j)=\sum_{j=0}^{[t]}\int_0^{t-j}f(t)\D t;
$$
thus conditions (A1) and (A2) are satisfied.}
\end{example}

\subsection{Absolute continuity of the time of capture $T(x)$}
Proposition~\ref{prop evol} gives sufficient conditions on the
birth-and-growth process for the existence of an evolution
equation for its mean densities. The condition of absolute
continuity of the time of capture $T(x)$ (defined in \eqref{T}) of
a given point $x\in\R^d$ is not trivial to check, in general. In
\cite{KV-surv} it is shown that if the mark distribution
$Q(t,\cdot\,)$ admits density on $\R^d$ for all $t>0$, then $T(x)$
is a continuous random variable, i.e.~$\pr(T(x)=t)=0$ for all
$t\in\R$; results about the absolutely continuity of $T(x)$ for
general birth-and-growth processes are not available in current
literature yet. In the following theorem we prove that for any
birth-and-growth process in $\cal G$, $T(x)$ is an absolutely
continuous random variable, i.e it admits a probability density
function.

\begin{theorem}\label{teo T}
For any birth-and-growth process in the class $\cal G$, the random
variable $T(x)$ admits probability density function for all
$x\in\R^d$.
\end{theorem}
\begin{proof}
By Besicovitch derivation theorem (see Theorem~2.22 in \cite{AFP})
we know that every  positive Radon measure $\eta$ on
$\mathbb{R}^d$ can be represented in the form
$\eta=\eta_{\ll}+\eta_{\perp},$ where $\eta_{\ll}$ and
$\eta_{\perp}$ are the absolutely continuous part of $\eta$ with
respect to $\h^d$ and the singular part of $\eta$, respectively,
and that $\eta_\perp$ is given by the restriction of $\eta$ to the
$\h^d$-negligible set \beq E=\Big\{y\in\R^d\,:\,\lim_{r\downarrow
0}\frac{\eta(B_r(y))}{\h^d(B_r(y))}=\infty\Big\},\label{set E}\eeq
where $B_r(y)$ is the
ball of radius $r$ centered in $y$.\\
Let $P^x$ be the probability measure on $\R$ of $T(x)$,
i.e.~$P^x(A):=\pr(T(x)\in A)$ for all Borel sets $A\subset\R$, and
observe that for all $t>0$
\begin{eqnarray}
P^x(B_{\Delta t}(t))&=&\pr(T(x)\in[t-\Delta t, t+\Delta t])\label{pr T}\\
&=&\pr(t<T(x)\leq t+\Delta t)+\pr(t-\Delta t<T(x)\leq
t)\nonumber\\
&=&\pr(\{N(\mathcal{C}(t+\Delta
t,x))>0\}\cap\{N(\mathcal{C}(t,x))=0\})\nonumber \\
&&\qquad\qquad+
\pr(\{N(\mathcal{C}(t,x))>0\}\cap\{N(\mathcal{C}(t-\Delta t,x))=0\})\nonumber\\
&=&\pr(\{N(\mathcal{C}(t+\Delta
t,x)\setminus\mathcal{C}(t,x))>0)\}\cap \{N(\mathcal{C}(t,x))=0\})\nonumber \\
&&\qquad\qquad+ \pr(\{N(\mathcal{C}(t,x)\setminus\mathcal{C}(
t-\Delta t,x))>0\}\cap \{N(\mathcal{C}(t-\Delta t,x))=0\})\nonumber \\
&\leq& \pr(N(\mathcal{C}(t+\Delta
t,x)\setminus\mathcal{C}(t,x))>0)+\pr(N(\mathcal{C}(t,x)\setminus\mathcal{C}(t-\Delta t,x))>0)\nonumber\\
&\leq& \E(N(\mathcal{C}(t+\Delta
t,x)\setminus\mathcal{C}(t,x))+\E(N(\mathcal{C}(t,x)\setminus\mathcal{C}(t-\Delta
t,x)))\nonumber\\
&=&\Lambda(\mathcal{C}(t+\Delta
t,x)\setminus\mathcal{C}(t,x))+\Lambda(\mathcal{C}(t,x)\setminus\mathcal{C}(t-\Delta
t,x)).\nonumber\end{eqnarray} Note that
$\mathcal{C}(s_1,x)\subset\mathcal{C}(s_2,x)$ for any $s_1,s_2$
with $s_1<s_2$, and, by assumption (A2), we know that
$\Lambda(\mathcal{C}(t,x))$ is continuously differentiable with
respect to $t$ with partial derivative given by equation
\eqref{evolvex} (with $\alpha(s,y)=\lambda(s)q(s,y)$). Then, being
$\h^1(B_{\Delta t}(t))=2\Delta t$, we get that for all $t>0$
\begin{eqnarray*}\limsup_{\Delta t\downarrow 0}\frac{P^x(B_{\Delta
t}(t))}{2\Delta t}&\stackrel{\eqref{pr T}}{\leq}&\limsup_{\Delta
t\downarrow 0}\frac{\Lambda(\mathcal{C}(t+\Delta
t,x))-\Lambda(\mathcal{C}(t,x))}{2\Delta t}+\limsup_{\Delta
t\downarrow
0}\frac{\Lambda(\mathcal{C}(t,x))-\Lambda(\mathcal{C}(t-\Delta
t,x))}{2\Delta t}\\
&=&\frac{1}{2}\Big(\frac{\partial}{\partial
t}^+\Lambda(\mathcal{C}(t,x))+\frac{\partial}{\partial
t}^-\Lambda(\mathcal{C}(t,x))\Big)\\&=&\frac{\partial}{\partial
t}\Lambda(\mathcal{C}(t,x))\stackrel{\eqref{evolvex}}{<}\infty.
\end{eqnarray*}
Thus we conclude that the set $E$ in\eqref{set E} is empty, and so
$P_\perp\equiv 0$, i.e.~$T(x)$ is an absolutely continuous random
variable.
\end{proof}

\subsection{Mean extended volume and surface densities and causal
cone}\label{subsec}
 Let us observe that
for any $d$-dimensional random closed set $\Xi$,
 by applying  Fubini's theorem (in $\Omega\times\mathbb{R}^d$,
with the product measure $\mathbb{P}\times\h^d$), we have
$$\mathbb{E}[\h^d(\Xi\cap A)]=\int_A \mathbb{P}(x\in\Xi) {\rm d} x
 \quad\forall A\in\mathcal{B}_{\mathbb{R}^d},
$$ and so, considering the birth-and-growth process
$\{\Theta^t\}_t$,
 we have that for any $t\in\R_+$
\begin{equation}V_V(t,x)=\mathbb{P}(x\in\Theta^t), \qquad
\h^d\mbox{-a.e. }x\in\R^d. \label{lim=P}\end{equation} Then it is
clear that
 Eq.~\eqref{nu-vex} may be
true for non-Poisson birth-and-growth process as well.  For
instance consider the process $\{\Theta^t\}_t$  in Example
\ref{esempio 1}; we know that in this case
 $V_{\rm
ex}(t,x)=V_V(t,x)$, so  for any $t\in\mathbb{R}_+$  the following
chain of equality holds for $\h^d$-a.e. $x\in\mathbb{R}^d$:
$$
V_{\rm
ex}(t,x)=\mathbb{P}(x\in\Theta^t)=\mathbb{P}(N(\mathcal{C}(t,x))>0)
=\mathbb{E}[N(\mathcal{C}(t,x))] =\Lambda(\mathcal{C}(t,x)).$$
Such relationship between $V_{\rm ex}$ and the causal cone is
proved  in \cite{BKM-hazard} in the Poissonian case, using the
fact that, since nuclei are assumed to be born accordingly with a
marked Poisson process,  for any fixed $t\in\mathbb{R}_+$ the
associated grains are independently and identically distributed as
a typical grain. We show here that Eq.~\eqref{nu-vex} holds for
any birth-and-growth process in $\cal G$, ans so reobtaining the
Poissonian case as special case by Proposition~\ref{prop poisson}.
 \begin{proposition}\label{teo vex} Let  $\{\Theta^t\}_t\in\mathcal{G}$.
 Then, for all $t\in\mathbb{R}_+$,
$$ V_{\rm ex}(t,x)=\Lambda(\mathcal{C}(t,x))
 \qquad \mbox{for }\h^d\mbox{-a.e.~}x\in\R^d.$$
\end{proposition}
\begin{proof}
Since $\Theta_{T_j}^t(X_j)=\emptyset$ if $T_j>t$, by the
definition of the mean extended volume measure in
Definition~\ref{def measures} we have that, for any fixed $t>0$,
$$\E[\mu_{\Theta^t}^{\rm ex}](\,\cdot\,)=\sum_j\E[\h^d(\Theta_{T_j}^t(X_j)\cap\cdot\,)], $$
and so its mean density $V_{\rm ex}(t,\cdot)$ is given by the sum
of the mean volume densities of each individual grain. Hence we
get that for $\h^d$-a.e.~$x\in\R^d$ \begin{multline*}V_{\rm
ex}(t,x)=\sum_j
\pr(x\in\Theta_{T_j}^t(X_j))=\sum_j\E[\1_{x\in\Theta_{T_j}^t(X_j)}]=
\sum_j\E[\1_{(T_j, X_j)\in\mathcal{C}(t,x)}]\\=\E[\sum_j\1_{(T_j,
X_j)\in\mathcal{C}(t,x)}]=\E[N(\mathcal{C}(t,x))]=\Lambda(\mathcal{C}(t,x)).
 \end{multline*}

\end{proof}

Now we are ready to state  the main result of this section, which
follows as a corollary of Theorem~\ref{teo T} and
Proposition~\ref{teo vex}.
 \begin{proposition}\label{corollary s-ex} For any birth-and-growth process $\{\Theta^t\}_t\in\mathcal{G}$
 the following equality for the mean extended surface density holds for all $t\in\R_+$: $$ S_{\rm
ex}(t,x)= \int_0^t  \int_{\partial {\cal S}_x(s,t)}
\lambda(s)q(s,y)\, \D K_{x,t,s}(y)\,\D s, \qquad \h^d\mbox{-a.e.
}x\in\R^d,$$ with $\D K_{x,t,s}(y)$ defined as in \eqref{K meas}.
\end{proposition}
\begin{proof}
By the definition of the class $\cal G$ and Theorem~\ref{teo T},
Proposition~\ref{prop evol} applies and so Eq.~\eqref{DVS} holds.
Then the assertion directly follows by Proposition~\ref{teo vex}
and Eq.~\eqref{evolvex}.
\end{proof}

\section{Final remarks}
From the previous sections we know that
\beq\frac{\partial}{\partial t}V_V(t,x)=\lim_{\Delta t\downarrow
0}\frac{\mathbb{P}(x\in \Theta^{t+\Delta t}\setminus
\Theta^t)}{\Delta t}=\lim_{\Delta t\downarrow
0}\frac{\mathbb{P}(N(\mathcal{C}(t+\Delta t,x)\setminus
\mathcal{C}(t,x))\geq 1\,\cap\,N( \mathcal{C}(t,x))=0)}{\Delta
t},\label{derivata Vv}\eeq while $$\frac{\partial}{\partial
t}V_{\rm ex}(t,x)=\lim_{\Delta t\downarrow
0}\frac{\Lambda(\mathcal{C}(t+\Delta t,x)\setminus
\mathcal{C}(t,x))}{\Delta t};$$ so, in general, $V_V$ cannot be
written in terms of $V_{\rm ex}$ only, except in the trivial case
in which only one nucleation may occur (see Example \ref{esempio
1}), and in the particular case of Poissonian nucleation process.
Indeed, if $N$ is a marked Poisson process, then it is well known
that it is a Poisson point process on the product space
$\R\!\times\R^d$, and so  the events $\{N(\mathcal{C}(t+\Delta
t,x)\setminus \mathcal{C}(t,x))\geq 1\}$ and $\{N(
\mathcal{C}(t,x))= 0\}$ are independent because
$[\mathcal{C}(t+\Delta t,x)\setminus \mathcal{C}(t,x)]\cap
\mathcal{C}(t,x)=\emptyset$; therefore by \eqref{derivata Vv} and
observing that $\mathbb{P}(N( \mathcal{C}(t,x))= 0)=1-V_V(t,x)$
and $$ \pr(N(\mathcal{C}(t+\Delta t,x)\setminus
\mathcal{C}(t,x))\geq 1)=\Lambda(\mathcal{C}(t+\Delta
t,x)\setminus \mathcal{C}(t,x))+o(\Delta t),$$  we reobtain
Eq.~\eqref{VV-Vex}.

Furthermore, we mention that three
 different kinds of nucleation can be considered in order to model various real situations.
\begin{enumerate}
\item \emph{Free nucleation}. Nuclei are allowed to be born also
in an already crystallized region; i.e.~if at a random time $T_j$
a new nucleation occurs, the probability law associated to $X_j$
does not depend on the space occupied  by $\Theta^{T_j}$, so that
$X_j$ may belong to $\Theta^{T_j}$.
 \item \emph{Thinned nucleation}. Nuclei which are born in an already crystallized region are
 removed;
i.e.~if nucleation occurs according to a \emph{free} process
$N_0$, then the considered nucleation process $N$ can be described
as a ``thinning'' (e.g., see \cite{stoyan}) of the MPP $N_0$.
Namely, in accordance to the previous notations,
 we have that if $N_0=\sum_j\delta_{T_j,X_j}$, then
$$N=\sum_j\delta_{T_j,X_j}(1-\1_{\bigcup_{i=1}^{j-1}\Theta_{T_i}^{T_j-}(X_i)}(X_j)). $$
We may notice that if $N_0$ is a marked Poisson point process with
intensity measure $\Lambda_0$, then the thinned process $N$ is not
Poissonian any longer, having intensity measure $\Lambda(\D
t\times \D x)= \Lambda_0(\D t\times \D
x)(1-\mathbb{P}(x\in\Theta^{t-}))$.
  \item \emph{New nuclei are forced to be
born in the free space $\mathbb{R}^d\setminus\Theta$.} Similarly
to the thinned process described above, in modelling real
applications sometimes it is assumed that new nuclei can be born
only in the free space; for instance consider the case in which
every new nucleation occurs in the free space uniformly in a
bounded region $A\subset\R^d$, i.e.~if a nucleation $X_j$ occurs
at time $T_j$, then $X_j$ is a random point uniformly distributed
in $A\setminus\Theta^{T_j-}$. It is clear that, in this case, the
probability distribution of every mark $X_j$ associated to $T_j$
depends on the whole history of the process, and in particular on
the crystallized region at time $T_j$.
\end{enumerate}
Throughout the paper we have considered a general nucleation
process $N$, so that our results apply making no distinction
between the three types of nucleation described above.   Note that
if $\{\Theta^t\}_t$ is a birth-and-growth process in $\mathcal{G}$
with free nucleation process $N_0$, then the birth-and-growth
process with thinned nucleation $N$ associated to $N_0$ belongs to
$\mathcal{G}$ as well. This might be useful whenever the free
nucleation process is simpler to handle.
\\

\noindent{\bf Acknowledgements}  Stimulating discussions on the
subject of the paper are acknowledged to
   V. Capasso and A. Micheletti of Dept. of Mathematics of the University of Milan.

\end{document}